\documentclass[11pt,reqno]{amsart}

\usepackage[utf8]{inputenc}
\usepackage[T1]{fontenc}
\usepackage{amsmath, amssymb, amsfonts, amsthm}
\usepackage{mathrsfs}
\usepackage{enumitem}
\usepackage{hyperref}
\usepackage{graphicx}
\usepackage[
paper=a4paper,
left=27mm,
right=27mm,
top=30mm,
bottom=30mm
]{geometry}
\usepackage{tikz}
\usetikzlibrary{decorations.markings, arrows.meta, calc}
\hypersetup{
	colorlinks = true,
	linkcolor  = blue,
	citecolor  = blue,
	urlcolor   = blue
}

\theoremstyle{plain}
\newtheorem{thm}{Theorem}[section]

\newtheorem{lem}[thm]{Lemma}

\theoremstyle{definition}

\theoremstyle{remark}
\newtheorem{remark}[thm]{Remark}

\title[Mazurov's question on fixed points of automorphisms]{Mazurov's question on fixed points of automorphisms of free Burnside groups}

\author{V.\,S.~Atabekyan}
\address{Department of Mathematics, Yerevan State University, Armenia}
\email{avarujan@ysu.am}

\date{\today}


\begin{document}
	
	\begin{abstract}
		In the Kourovka Notebook, V.\,D.~Mazurov asked whether an automorphism of prime order $m$ of the infinite free Burnside group $B(m,q)$ of prime period $q$ that cyclically permutes the free generators must have a nontrivial fixed point. We show that the answer is negative for all sufficiently large prime periods $q$. In fact, we prove a slightly stronger result in which, for $m\ge3$, the primality of $m$ is replaced by the weaker assumption that $m$ is odd: for every odd $m\ge3$ there exists $q_m>0$ such that, for every prime $q>q_m$, the automorphism of $B(m,q)$ of order $m$ that cyclically permutes the free generators is fixed-point-free.
	\end{abstract}
	
	\maketitle
	
	\section{Introduction}
	
	In the Kourovka Notebook \cite{KN}, V.\,D.~Mazurov posed the following question (Problem~17.70): \textit{Let $\alpha$ be an automorphism of prime order $m$ of the infinite free Burnside group $G=B(m,q)$ of prime period $q$ that cyclically permutes the free generators of $G$. Is it true that $\alpha$ has fixed points?}
	
	In connection with this question, nontrivial fixed points were constructed explicitly in \cite{AtAs19} for every automorphism of order~$2$ of a nonabelian periodic group of odd period~$q$. For example, if $\beta(a)=b$ and $\beta(b)=a$, then $a(a^{-1}b)^{\frac{q+1}{2}}$ is a nontrivial fixed point of $\beta$ (see \cite{AtAs19}). Thus, for $m=2$ the answer to Mazurov's question is positive.
	
	The aim of this paper is to show that for $m\ge3$ the answer is negative for all sufficiently large prime periods $q$. More precisely, we prove the following theorem.
	
	\begin{thm}\label{mq}
		For every odd $m\ge3$ there exists $q_m>0$ such that, for every prime $q>q_m$, the automorphism of order $m$ of the free Burnside group $B(m,q)$ that cyclically permutes the free generators of $B(m,q)$ has no nontrivial fixed points.
	\end{thm}
	
	Recall that the free Burnside group $B(m,n)$ of rank $m$ and period $n$ is defined by the presentation
	$$B(m,n)=\langle x_1, x_2, \ldots, x_m \mid X^n=1\rangle,$$
	where $X$ ranges over all words in the alphabet $\{x_1^{\pm1},x_2^{\pm1},\ldots,x_m^{\pm1}\}$. In other words, $B(m,n)$ is the quotient of the free group $F_m=\langle x_1,x_2,\ldots,x_m\rangle$ of rank $m$ by the normal subgroup $F_{m}^n$ generated by the $n$th powers of all elements of $F_m$. By a theorem of S.\,I.~Adian \cite{A}, the groups $B(m,n)$ are infinite for all odd $n\ge665$ and all $m>1$.
	
	Free Burnside groups of sufficiently large odd period share many properties with absolutely free groups. For instance: all their abelian and all their finite subgroups are cyclic; the centralizer of every nontrivial element is cyclic; the word and conjugacy problems are solvable; the groups grow exponentially (Adian \cite{A}); they are non-amenable (Adian \cite{A82}) and, in fact, uniformly non-amenable (Osin \cite{Os}; see also \cite{At09u}, \cite{At09m}); and they are approximated by free groups of smaller rank, among other shared properties (see \cite{O}, \cite{At87}, \cite{At07}).
	
	There are, however, fundamental differences between the two classes of groups as well. Every subgroup of a free group is itself free (the Schreier property), whereas A.\,Yu.~Ol'shanskii \cite{O03} showed that no proper normal subgroup of $B(m,n)$ is a free Burnside group (see also \cite{At10}). Nevertheless, every noncyclic subgroup of a free Burnside group contains a free Burnside group of infinite rank (see \cite{O}, \cite{At87}, \cite{At09i}, \cite{I}, \cite{Sh}).
	
	As for automorphisms of $B(m,n)$, it is known that for every odd $n\ge1003$ all normal automorphisms of $B(m,n)$, as well as all its splitting automorphisms of order $p^k$ ($p$ a prime), are inner (see \cite{Ch}--\cite{At14}). It is also known that for the same values of $n$, the automorphism group $\mathrm{Aut}(B(m,n))$ is complete \cite{At13p, At13ij} and that $\mathrm{Aut}(\mathrm{End}(B(m,n))) \simeq \mathrm{Aut}(\mathrm{Aut}(B(m,n))) \simeq \mathrm{Aut}(B(m,n))$ \cite{At15}; the same properties hold for the automorphism group $\mathrm{Aut}(F_m)$ of the free group $F_m$ (see \cite{F75}, \cite{F02}). Moreover, for all sufficiently large odd $n$, the outer automorphism group $\mathrm{Out}(B(m,n))$ contains non-abelian free subgroups (Coulon \cite{Cou}), as does $\mathrm{Out}(F_m)$. Theorem~\ref{mq} reveals yet another analogy between absolutely free groups and free Burnside groups of sufficiently large odd period.
	
	To prove Theorem~\ref{mq}, in Section~\ref{p2} we construct periodic groups of odd period $mpq$, where $m\ge3$ is odd and $p>m$, $q>pm$ are sufficiently large primes. The construction relies heavily on the well-known monograph of A.\,Yu.~Ol'shanskii \cite{O}. Theorem~\ref{mq} will be derived in Section~\ref{pmq} from the following result, which is of independent interest.
	
	\begin{thm}\label{mpq}
		Let $m\ge3$ be odd, and let $p>m$ and $q>pm$ be sufficiently large primes. Then there exists a two-generator periodic group $G=\langle a,b\rangle$ of period $mpq$ such that
		\begin{enumerate}
			\item the generator $a$ has order $mp$;
			\item the normal closure $\langle b\rangle^G$ of $b$ in $G$ is a periodic group of period $q$;
			\item the inner automorphism $i_{a}$ acts on $\langle b\rangle^G$ without nontrivial fixed points;
			\item $\langle b\rangle^G$ contains a free Burnside group $H$ of rank $m$;
			\item conjugation by $a^p$ cyclically permutes the free generators of $H$.
		\end{enumerate}
	\end{thm}
	
	The group $G$ of Theorem~\ref{mpq} is constructed in Section~\ref{p2}, and the proof of Theorem~\ref{mpq} occupies Sections~\ref{p2}--\ref{pmpq}. In Section~\ref{p5} we establish some basic properties of $G$; in particular, we prove items~1 and~2 of Theorem~\ref{mpq}. In Section~\ref{p3} we introduce words in block form, which represent elements of a certain subgroup of the free product of two cyclic groups, and in Section~\ref{p4.1} we prove the key property of such words representing the identity of $G$ (Lemma~\ref{rw}). Item~4 of Theorem~\ref{mpq} is proved in Section~\ref{p4}, and the proof of Theorem~\ref{mpq} is completed in Section~\ref{pmpq}. Finally, Theorem~\ref{mq} is proved in Section~\ref{pmq}.
	
	\section{Construction of the periodic group}\label{p2}
	
	We fix the group alphabet $\{a^{\pm1},b^{\pm1}\}$ and denote by $G(0) = F_2$ the free group with basis $\{a,b\}$. We also fix an odd number $m\ge3$ and sufficiently large primes $p>m$ and $q>pm$.
	
	In this section we construct a group that is a simplified version of the group $G(\infty)$ introduced by A.\,Yu.~Ol'shanskii in \cite[Ch.~8, \S25.1]{O}. We reproduce the definition of this construction below, with the necessary modifications; to avoid repetition, we shall not refer to \cite[Ch.~8, \S25.1]{O} again within this section.
	
	Set $\mathcal{R}_0=\emptyset$ and $\mathcal{P}_1 = \{a^{n_a}, b^{n_b}\}$, where $n_a=mp$ and $n_b=q$, and let $\mathcal{R}_1=\mathcal{R}_0\cup\mathcal{P}_1$. The words $a$ and $b$ are called periods of rank~$1$. Now fix $i\ge2$ and suppose, by induction, that the set $\mathcal{R}_{i-1}$ of defining relations has already been defined. Set $G(i-1) = \langle a, b \mid R = 1,\ R \in \mathcal{R}_{i-1} \rangle$. Suppose also that for each $j<i$ we have defined a set $\mathcal{X}_j$, whose elements are called periods of rank~$j$. A word $X$ is called minimal in rank $i-1$ if $X = Y$ in $G(i-1)$ implies $|X| \le |Y|$. A nonempty word $A$ is called simple in rank $i-1$ if it is conjugate in rank $i-1$ (that is, in $G(i-1)$) neither to a power of a shorter word nor to a power of a period of rank $k \le i-1$.
	
	Let $\mathcal{X}_i$ be a maximal set of words of length $i$ that are simple in rank $i-1$ and have the following property: if $A, B \in \mathcal{X}_i$ and $A \neq B$, then $A$ is conjugate in rank $i-1$ neither to $B$ nor to $B^{-1}$. The words in $\mathcal{X}_i$ are called periods of rank~$i$. Using the identity $yx=xyy^{-1}x^{-1}yx=xy[y,x]$, every word $A\in\mathcal{X}_i$ can be written uniquely in the form $A=a^kb^sc$, where $k,s\in \mathbb{Z}$ and $c$ belongs to the commutator subgroup $[F_2,F_2]$ of $F_2$.
	
	The set $\mathcal{P}_i$ of defining words of rank $i\ge2$ is formed as follows. Let a period $A\in \mathcal{X}_i$ be written in the free group $G(0)=\langle a,b\rangle$ as $A=a^kb^sc$, where $c$ is a commutator word. Since $a^kb^sc$ is simple in rank $i-1\ge1$ and $a^{mp}=1$, $b^q=1$ in rank~$1$, we may assume that $0\le k<mp$ and $0\le s<q$. For each period $A=a^kb^sc \in \mathcal{X}_i$ we include in $\mathcal{P}_i$ the word $A^{n_A}$, where $n_A=\frac{mpq}{(k,mp)}$ and $(k,mp)$ denotes the greatest common divisor of $k$ and $mp$. The relations $A^{\frac{mpq}{(k,mp)}}=1$ are called defining relations of rank $i\ge2$. In particular, if $k=0$, that is, $A=b^sc \in \mathcal{X}_i$, then $(k,mp)=(0,mp)=mp$, and hence $A^q\in \mathcal{P}_i$. Words of the form $A=b^sc$ are called words of type~$(b)$. It is easy to see that words of type $(b)$ are precisely the words representing elements of the normal closure $\langle b\rangle^{G(0)}$ of $b$ in the free group $G(0)$.
	
	All words in $\mathcal{X}_i$ are simple in rank $i-1$ and have length~$i$. Therefore, if $X=a^kb^sc\in\mathcal{X}_i$ with $i\ge2$, then $X\notin\langle a\rangle$ and $X\notin\langle b\rangle$.
	
	Finally, we set $\mathcal{R}_{i} = \mathcal{R}_{i-1} \cup \mathcal{P}_i$, $G(i) = \langle a,b \mid R = 1,\ R \in \mathcal{R}_{i} \rangle$, and
	\begin{equation}\label{e1}
		G = \Big\langle a,b \;\Big|\; R = 1,\ R \in \mathcal{R} = \bigcup_{i=0}^\infty \mathcal{R}_i \Big\rangle.
	\end{equation}
	
	By construction, $G$ is a simplified version of the group $G(\infty)$ from \cite[Ch.~8, \S25]{O}: it is obtained by omitting the relations of the second type (relations~(2) in \cite[Ch.~8, \S25]{O}). Consequently, all results proved in \S\S25 and~26 of Chapter~8 of \cite{O} (Lemmas 25.1--25.21 and 26.1--26.5, and Theorems 26.1--26.5) remain valid for the group $G$ defined by~\eqref{e1}.
	
	For the reader's convenience, we state below the versions of Theorems 26.4 and 26.5 of Chapter~8, Theorems 13.1 and 16.2, and Lemma 19.4 of \cite{O}, adapted to the group $G$; they will be used in Sections~\ref{p5}, \ref{p4.1} and~\ref{p4}.
	
	\begin{lem}[{cf. \cite[Theorem~26.4]{O}}]\label{26.4}
		\begin{enumerate}
			\item Every period $A$ has order $n_A$ in $G$.
			\item If each set $\mathcal{X}_i$, $i\ge1$, is maximal, then every element of $G$ is conjugate in $G$ to a power of a period of some rank.
			\item If $A$ is a period of some rank, then $\langle A\rangle$ is not contained in any larger cyclic subgroup of $G$.
		\end{enumerate}
	\end{lem}
	
	\begin{lem}[{cf. \cite[Theorem~26.5]{O}}]\label{26.5}
		The centralizer of every nontrivial element of $G$ is cyclic.
	\end{lem}
	
	\begin{lem}[{cf. \cite[Theorem~13.1]{O}}]\label{13.1}
		A word $W$ is equal to $1$ in the group $G$ given by the graded presentation \eqref{e1} if and only if there exists a reduced graded disk diagram over \eqref{e1} whose contour label is graphically equal to $W$.
	\end{lem}
	
	\begin{lem}[{cf. \cite[Lemma~19.4]{O}}]\label{19.4}
		Every reduced diagram is an $A$-map.
	\end{lem}
	
	\begin{lem}[{cf. \cite[Theorem~16.2]{O}}]\label{16.2}
		If $\Delta$ is an $A$-map of nonzero rank, then $\Delta$ contains a cell $\Pi$ and a contiguity submap $\Gamma$ of $\Pi$ to a section $w$ of the contour of $\Delta$ such that $r(\Gamma)=0$ and $(\Pi, \Gamma, w)\ge\varepsilon$.
	\end{lem}
	
	\section{Basic properties of the group $G$}\label{p5}
	
	\begin{lem}\label{l1}
		The element $a$ has order $mp$ in $G$, and every other period $A=a^kb^sc\in\mathcal{X}_i$ has order $\frac{mpq}{(k,mp)}$. In particular, every period of type $(b)$ has order $q$.
	\end{lem}
	\begin{proof}
		By Lemma~\ref{26.4}(1), if $A$ is a period of rank $i$, that is, $A\in\mathcal{X}_i$ and $A^{n_A}\in \mathcal{P}_i$, then $A$ has order $n_A$ in $G$. In our case $a^{mp}, b^q\in \mathcal{P}_1$, while for every other period $A=a^kb^sc$ we have $A^{\frac{mpq}{(k,mp)}}\in \mathcal{P}_i$, $i\ge2$. Hence $a$ has order $mp$ in $G$, and every other period has order $\frac{mpq}{(k,mp)}$. In particular, every period $A=b^sc$ of type $(b)$ has order $\frac{mpq}{(0,mp)}=q$.
	\end{proof}
	
	\begin{lem}\label{l2}
		$\langle a\rangle\cap\langle b\rangle^G=\{1\}$, where $\langle b\rangle^G$ denotes the normal closure of $b$ in $G$.
	\end{lem}
	\begin{proof}
		Adding the relation $b=1$ to the defining relations of $G$, we obtain the group
		$$\Gamma=\Big\langle a,b \;\Big|\; b=1,\ R = 1,\ R \in \mathcal{R} = \bigcup_{i=0}^\infty \mathcal{R}_i \Big\rangle.$$
		Every relation $A^q=1$ of $\Gamma$, where $A=b^sc$ is a defining word of type $(b)$, is a consequence of $b=1$, since $c$ is a commutator word in $a$ and $b$. Every other relation $(a^kb^sc)^{\frac{mpq}{(k,mp)}}=1$ is a consequence of the relations $a^{mp}=1$ and $b=1$. Therefore
		$$\Gamma=\langle a,b\mid a^{mp}=1,\ b=1 \rangle=\langle a\mid a^{mp}=1\rangle,$$
		that is, $\Gamma$ is a cyclic group of order $mp$. On the other hand, $\Gamma\cong G/\langle b\rangle^G$, and this quotient is generated by the image of $a$. If $\langle a\rangle\cap\langle b\rangle^G\neq\{1\}$, the image of $a$ in $G/\langle b\rangle^G$ would have order less than $mp$, and hence $|\Gamma|<mp$, a contradiction.
	\end{proof}
	
	\begin{lem}\label{l3}
		The normal subgroup $\langle b\rangle^G$ of $G$ is a periodic group of period $q$.
	\end{lem}
	\begin{proof}
		Since each set $\mathcal{X}_i$ is maximal, Lemma~\ref{26.4}(2) shows that every nontrivial element $x\in G$ is conjugate in $G$ to a power of a period of some rank $j$. Let $x\neq1$ be an element of $\langle b\rangle^G$, and let $x=UA^tU^{-1}$, where $A$ is a period of rank $j$. Then $A^t\in\langle b\rangle^G$.
		
		If $j=1$, then $A\in\{a,b\}$. Since $A^t\in\langle b\rangle^G$ and $\langle a\rangle\cap\langle b\rangle^G=\{1\}$ by Lemma~\ref{l2}, we have $A\neq a$. Hence $A=b$, and $x$ has order $q$.
		
		Now let $j\ge2$, and write $A=a^kb^sc$ in the free group $F_2$, where $c$ is a commutator word and $a^kb^sc$ is simple in rank $j-1\ge1$. Since $b^s\in\langle b\rangle^G$ and $c\in\langle b\rangle^G$, it follows that $a^{kt}\in\langle b\rangle^G$. By Lemma~\ref{l2}, $a^{kt}=1$ in $G$, and hence $mp\mid kt$ by Lemma~\ref{l1}. Let $l=(k,mp)$. Then $t=\frac{mp}{l}t_1$ for some integer $t_1$. By Lemma~\ref{l1}, the word $A=a^kb^sc$ has order $\frac{mpq}{l}$ in $G$. Therefore the element $A^t=A^{\frac{mp}{l}t_1}\neq1$ has prime order $q$ in $G$, and so does $x=UA^{t}U^{-1}$.
	\end{proof}
	
	\begin{remark}
		A short additional argument shows that $G$ has another interesting property: it is isomorphic to the strictly Burnside $q$-product, introduced by S.\,V.~Ivanov \cite{Is}, of a cyclic group of order $mp$ and a cyclic group of order $q$ (see also \cite[Ch.~11, \S36]{O}).
	\end{remark}
	
	\section{Words in block form}\label{p3}
	
	In this section we introduce a class of words in the alphabet $\{a^{\pm1}, b^{\pm1}\}$ that will be used in Sections~\ref{p4.1} and~\ref{p4} to construct a free Burnside subgroup of rank $m$ in the normal subgroup $\langle b\rangle^G$ of the group $G$ defined by~\eqref{e1}. For a word $U$, we denote by $\sigma_a(U)$ the exponent sum of the letter $a$ in $U$; clearly, $\sigma_a$ induces a homomorphism from the free group $\langle a,b\rangle$ onto $\mathbb{Z}$.
	
	Consider the absolutely free group $G(0)=\langle a,b\rangle$ of rank~$2$ and set $y=b^{-1}a^{-1}b^2ab^{-1}$. We define the following words inductively:
	\begin{equation}\label{xyz}
		y_1=y=b^{-1}a^{-1}b^2ab^{-1},\quad
		y_2=a^{p}y_1a^{-p}, \ldots,\quad y_m=a^{p}y_{m-1}a^{-p},
	\end{equation}
	where $p$ is the prime fixed in Section~\ref{p2}. Thus,
	\begin{equation}\label{xyzy}
		y_1=y,\quad
		y_2=a^{p}ya^{-p}, \ldots,\quad y_m=a^{(m-1)p}ya^{-(m-1)p}.
	\end{equation}
	The elements $y_1,y_2,\ldots, y_m$ in \eqref{xyzy} freely generate a free subgroup $H_1=\langle y_1,y_2,\ldots, y_m\rangle$ of rank $m$ in $\langle a,b\rangle$, since the sequence $y_1,y_2,\ldots, y_m$ is Nielsen reduced. Hence the map
	\begin{equation}\label{perm}
		\left\{ \begin{array}{rcl}
			x_1&\mapsto& y_1 \\ x_2&\mapsto& y_2 \\ \vdots & &\vdots\\ x_m&\mapsto&  y_m\end{array}\right.
	\end{equation}
	extends uniquely to an embedding
	\begin{equation}\label{alpha}
		\alpha\colon F_m=\langle x_1,x_2,\ldots, x_m\rangle\to \langle a,b\rangle
	\end{equation}
	of the absolutely free group $F_m$ of rank $m$ into the free group $\langle a,b\rangle$ of rank~$2$, with $\mathrm{Im}(\alpha)=H_1$. Since $\sigma_a(y_i)=0$ for all $i$, we have $\sigma_a(\alpha(W))=0$ for every $W\in F_m$.
	
	In parallel, consider the free product
	\begin{equation}\label{g1}
		G_1=\mathbb{Z}_{mp}\ast\mathbb{Z}
	\end{equation}
	of the cyclic group $\mathbb{Z}_{mp}=\langle a \mid a^{mp}=1 \rangle$ of order $mp$ generated by $a$ and the infinite cyclic group $\mathbb{Z}=\langle b \rangle$ generated by $b$. In $G_1$ we consider the same words
	\begin{equation}\label{xyzy*}
		y_1=y=b^{-1}a^{-1}b^2ab^{-1},\quad
		y_2=a^{p}ya^{-p}, \ldots,\quad y_m=a^{(m-1)p}ya^{-(m-1)p}.
	\end{equation}
	For every word $W\in F_m$, the word $\alpha(W)$ also represents an element of $G_1$.
	
	Every element of $G_1$ has a normal form in which syllables from $\mathbb{Z}_{mp}$ alternate with syllables from $\mathbb{Z}$. If $W$ is reduced, then, when the neighbouring factors $y_1^{\pm1}, y_2^{\pm1}, \ldots, y_m^{\pm1}$ of $\alpha(W)$ are multiplied, the middle part $(a^{-1}b^2a)^{\pm1}$ of each of them is never cancelled. Therefore the normal form of $\alpha(W)$ has length at least $3|W|$, where $|W|$ is the length of $W$ in $F_m$. Consequently, every nontrivial word $W$ is mapped to a word $\alpha(W)$ representing a nontrivial element of $\mathbb{Z}_{mp} * \mathbb{Z}$. Thus the composition
	\begin{equation}\label{alphabeta}
		\beta\circ\alpha\colon F_m \to G_1
	\end{equation}
	is an embedding, where
	\begin{equation}\label{beta}
		\beta\colon\langle a,b\rangle\to G_1
	\end{equation}
	is the natural epimorphism sending each word of the free group $\langle a,b\rangle$ to the element of $G_1$ that it represents. We set
	$$H_2=\beta(H_1)=\beta(\alpha(F_m))=\langle y_1,\ldots,y_m\rangle\le G_1;$$
	thus $\beta\circ\alpha$ is an isomorphism of $F_m$ onto $H_2$.
	
	Note that for every integer $u\neq0$,
	\begin{align*}
		y^u &= [(b^{-1}a^{-1}b^2ab^{-1})(b^{-1}a^{-1}b^2ab^{-1})\cdots (b^{-1}a^{-1}b^2ab^{-1})]^{\operatorname{sgn}(u)}\\
		&= [b^{-1}(a^{-1}b^2ab^{-2})(a^{-1}b^2ab^{-2})\cdots (a^{-1}b^2ab^{-2})(a^{-1}b^2ab^{-1})]^{\operatorname{sgn}(u)}\\
		&= [b^{-1}(a^{-1}b^2ab^{-2})^{|u|-1}a^{-1}b^2ab^{-1}]^{\operatorname{sgn}(u)}.
	\end{align*}
	Hence, for $u\neq0$, the word
	\begin{equation}\label{nfx}
		y^u=[b^{-1}(a^{-1}b^2ab^{-2})^{|u|-1}a^{-1}b^2ab^{-1}]^{\operatorname{sgn}(u)}
	\end{equation}
	is the reduced form of $y^u$ both in the free group $\langle a,b\rangle$ and in $G_1$. In particular, it begins and ends with a letter $b^{\pm1}$.
	
	Let $K=\langle a^p,y\rangle$ be the subgroup of the free group $\langle a,b\rangle$ generated by $a^p$ and $y$. Since no cancellation occurs between a nontrivial power of $a^p$ and a nontrivial power of $y$, the group $K$ is free with basis $\{a^p,y\}$, and the reduced form of every element of $K$ is a word
	\begin{equation}\label{U}
		a^{t_1p}\cdot y^{u_1}\cdot a^{t_2p}\cdot y^{u_2}\cdots a^{t_kp}\cdot y^{u_k}\cdot a^{t_{k+1}p},
	\end{equation}
	where $k\ge0$, the integers $u_1,\dots,u_k$ and $t_2,\dots,t_k$ are nonzero, $t_1$ and $t_{k+1}$ are arbitrary integers, and each $y^{u_i}$ is the word \eqref{nfx}; conversely, every word of this form is reduced. A word of the form \eqref{U} is said to be in \textbf{block form}; its subwords $a^{t_ip}$ and $y^{u_i}$ are called $a$-blocks and $y$-blocks, respectively. Thus the words in block form are exactly the reduced forms of the elements of $K$, and we shall not distinguish between an element of $K$ and its block form. Since $y_1,\ldots,y_m\in K$, the reduced form of $\alpha(W)$ is in block form for every $W\in F_m$.
	
	The first and the last letters of a word \eqref{U} are determined by its first and last nontrivial blocks, and they are mutually inverse if and only if these blocks are $a^{tp}$ and $a^{t'p}$ with $tt'<0$, or $y^{u}$ and $y^{u'}$ with $uu'<0$. Hence an element of $K$ is cyclically reduced with respect to the basis $\{a^p,y\}$ if and only if its block form is a cyclically reduced word. In particular, every element of $K$ is conjugate in $K$ to an element whose block form is a cyclically reduced word.
	
	Let $U$ be a word in block form. Regarding each $y$-block $y^{u_i}$ as the product of $|u_i|$ factors $y^{\operatorname{sgn}(u_i)}$ and each $a$-block as the product of its letters, we write $U$ as a product of factors of the form $a^{\pm1}$ and $y^{\pm1}$; the factors $y^{\pm1}$ are called the $y$-factors of $U$. If $U$ is cyclically reduced, then the cyclic word $U$ is also a cyclic product of these factors.
	
	\begin{lem}\label{align}
		Let $U$ be a cyclically reduced word in block form, regarded as a cyclic word. Then:
		\begin{enumerate}
			\item every letter $b^{\pm1}$ of $U$ belongs to a $y$-factor of $U$;
			\item every $a$-syllable of $U$ (that is, every maximal subword of the form $a^{e}$) either has length $1$ and lies in a $y$-factor, or has length at least $p$ and is an $a$-block of $U$ or the union of the last and the first $a$-blocks of $U$;
			\item every $b$-syllable of $U$ has length at most $2$;
			\item every occurrence of the word $y$ or $y^{-1}$ in $U$ is a $y$-factor of $U$.
		\end{enumerate}
	\end{lem}
	\begin{proof}
		(1) This is clear, since $a$-blocks contain no letters $b^{\pm1}$.
		
		(2) Each $y$-factor $y^{\pm1}$ begins and ends with a letter $b^{\pm1}$ and contains exactly two letters $a^{\pm1}$, namely $a^{-1}$ in the second position and $a$ in the fifth one (recall that $y^{-1}=ba^{-1}b^{-2}ab$). Hence an $a$-syllable of $U$ either is a single letter of a $y$-factor, or consists of letters of $a$-blocks. In the latter case it is an $a$-block, or, in the cyclic word, the union of the last and the first $a$-blocks; its length is at least $p$, since $t_i\neq0$ for $2\le i\le k$ and, $U$ being cyclically reduced, $t_1$ and $t_{k+1}$ are not of opposite signs.
		
		(3) The $b$-syllables of a single $y$-factor are $b^{\pm1}$ and $b^{\pm2}$. Two $y$-factors are adjacent only inside a $y$-block, or cyclically if $U$ begins and ends with $y$-blocks; in both cases they are equal (since $U$ is cyclically reduced), and their neighbouring letters merge into $b^{-2}$ in $yy$ and into $b^{2}$ in $y^{-1}y^{-1}$.
		
		(4) Let $x$ be an occurrence of $y=b^{-1}a^{-1}b^2ab^{-1}$ in $U$. The second letter of $x$ is $a^{-1}$, and it is preceded and followed by letters $b^{\pm1}$. By (2), it is the second letter of some $y$-factor $f$ of $U$. The third and fourth letters of $x$ are $b^2$, whereas the third and fourth letters of $y^{-1}$ are $b^{-2}$; hence $f=y$, and since $x$ and $f$ begin at the same position, $x=f$. The case of an occurrence of $y^{-1}$ is similar.
	\end{proof}
	
	\begin{lem}\label{crit}
		Let $L=\langle a^p,y\rangle$ be the subgroup of $G_1$ generated by $a^p$ and $y$. Then:
		\begin{enumerate}
			\item $H_2$ is a normal subgroup of $L$;
			\item a word $V$ in block form represents an element of $H_2$ if and only if $\sigma_a(V)\equiv0\pmod{mp}$.
		\end{enumerate}
	\end{lem}
	\begin{proof}
		(1) Since $a^{mp}=1$ in $G_1$, conjugation by $a^p$ permutes the elements $y_1,\ldots,y_m$ in \eqref{xyzy*} cyclically; moreover, $y=y_1\in H_2$. Hence $H_2=\langle y_1,\ldots,y_m\rangle$ is invariant under conjugation by $a^{\pm p}$ and $y^{\pm1}$, and so it is normal in $L$.
		
		(2) Let $\rho\colon G_1\to\langle a\rangle$ be the homomorphism onto the cyclic group of order $mp$ that is identical on $a$ and sends $b$ to $1$. Then $\rho(H_2)=1$ and $\rho(\beta(V))=a^{\sigma_a(V)}$; hence $\beta(V)\in H_2$ implies $mp\mid\sigma_a(V)$. Conversely, by (1) and since $y\in H_2$, the quotient $L/H_2$ is a cyclic group generated by the image of $a^p$. If $V$ has the form \eqref{U}, then $\beta(V)\in L$ and $\beta(V)\equiv(a^p)^{t_1+\dots+t_{k+1}}=a^{\sigma_a(V)}$ modulo $H_2$. Therefore, if $mp\mid\sigma_a(V)$, then $\beta(V)\in H_2$, because $a^{mp}=1$ in $G_1$.
	\end{proof}
	
	Let $K_0=\{V\in K\mid \sigma_a(V)\equiv0\pmod{mp}\}$; this is a normal subgroup of $K$. Denote by $F_m^q$ the normal subgroup of $F_m$ generated by the $q$th powers of all its elements. It is a characteristic subgroup of $F_m$, and $F_m/F_m^q=B(m,q)$.
	
	\begin{lem}\label{omega}
		\begin{enumerate}
			\item The map $\omega\colon K_0\to F_m$ given by $\omega(V)=(\beta\circ\alpha)^{-1}(\beta(V))$ is a well-defined homomorphism. Moreover, $\omega(\alpha(W))=W$ for every $W\in F_m$, and $\omega(a^{mp})=1$.
			\item If $V\in K_0$ and $Z\in K$, then $ZVZ^{-1}\in K_0$, and $\omega(V)\in F_m^q$ if and only if $\omega(ZVZ^{-1})\in F_m^q$.
		\end{enumerate}
	\end{lem}
	\begin{proof}
		(1) By Lemma~\ref{crit}(2), $\beta(K_0)\subseteq H_2$, and $\beta\circ\alpha$ is an isomorphism of $F_m$ onto $H_2$. Hence $\omega$ is a well-defined homomorphism. Since $\sigma_a(\alpha(W))=0$, we have $\alpha(W)\in K_0$ and $\omega(\alpha(W))=W$. Finally, $\beta(a^{mp})=1$, whence $\omega(a^{mp})=1$.
		
		(2) Clearly, $\sigma_a(ZVZ^{-1})=\sigma_a(V)$, so $ZVZ^{-1}\in K_0$. By Lemma~\ref{crit}(1), conjugation by $\beta(Z)\in L$ maps $H_2$ onto itself; hence $\omega(ZVZ^{-1})=\theta(\omega(V))$, where $\theta=(\beta\circ\alpha)^{-1}\circ\iota_{\beta(Z)}\circ(\beta\circ\alpha)$ is an automorphism of $F_m$ and $\iota_g(h)=ghg^{-1}$. It remains to note that $F_m^q$ is a characteristic subgroup of $F_m$. (If $Z\notin K_0$, then $\theta$ need not be inner: for instance, conjugation by $a^p$ induces the automorphism $x_i\mapsto x_{i+1}$, $x_m\mapsto x_1$, of $F_m$; so the normality of $F_m^q$ alone does not suffice.)
	\end{proof}
	
	\section{Words in block form representing the identity of $G$}\label{p4.1}
	
	As in Section~\ref{p2}, we assume that $m\ge3$ is odd and that $p>m$ and $q>pm$ are sufficiently large primes. By the least parameter principle (see \cite[\S15.1]{O}), we may also assume that the constant $\varepsilon$ in Lemma~\ref{16.2} satisfies
	\begin{equation}\label{epsilon}
		\varepsilon mp>2\quad\text{and}\quad\varepsilon q>7.
	\end{equation}
	
	For a diagram $\Delta$ over the presentation \eqref{e1}, let $\nu_j(\Delta)$ be the number of cells of rank $j$ in $\Delta$, and call the sequence $\tau(\Delta)=(\nu_j(\Delta))_{j\ge1}$ the \emph{type} of $\Delta$. We write $\tau(\Delta)<\tau(\Delta')$ if $\nu_j(\Delta)<\nu_j(\Delta')$ for the largest $j$ with $\nu_j(\Delta)\neq\nu_j(\Delta')$; this is a well-ordering. If a word $V$ represents the identity of $G$, then by Lemma~\ref{13.1} there are diagrams with contour label $V$, and we denote by $\tau(V)$ the least type of such a diagram. A diagram of type $\tau(V)$ with contour label $V$ is reduced since the cancellation of a $j$-pair decreases the type (see \cite[\S13.2]{O}). Moreover, $\tau(V)$ does not change under cyclic shifts of $V$ and does not increase under free and cyclic reductions of $V$, since folding the contour of a diagram does not change its cells.
	
	\begin{lem}\label{rw}
		Let $U$ be a word in block form that represents the identity of $G$. Then $U\in K_0$ and $\omega(U)\in F_m^q$; in other words,
		\begin{equation}\label{W}
			\omega(U)=V_{1}B_1^qV_{1}^{-1}\cdots V_kB_k^qV_k^{-1}
		\end{equation}
		for some words $V_1,\ldots, V_k,B_1,\ldots,B_k\in F_m$.
	\end{lem}
	\begin{proof}
		By the proof of Lemma~\ref{l2}, the quotient $G/\langle b\rangle^G$ is a cyclic group of order $mp$ generated by the image of $a$, and the image of $U$ in it is $a^{\sigma_a(U)}$. Since $U=1$ in $G$, we obtain $\sigma_a(U)\equiv0\pmod{mp}$, that is, $U\in K_0$.
		
		By Section~\ref{p3}, $U=ZU_0Z^{-1}$ for some $Z\in K$ and some $U_0\in K$ whose block form is cyclically reduced. Then $U_0=1$ in $G$, and $\omega(U)\in F_m^q$ if and only if $\omega(U_0)\in F_m^q$ by Lemma~\ref{omega}(2). Therefore it suffices to prove the lemma for cyclically reduced words $U$ in block form, and we do this by induction on $\tau(U)$.
		
		Let $\Delta$ be a diagram of type $\tau(U)$ whose contour $w$ has label $\varphi(w)=U$; it is reduced. If $\Delta$ has no cells of positive rank, then $U$ is freely trivial; since $U$ is reduced, it is the empty word, and $\omega(U)=1$. Otherwise, by Lemma~\ref{19.4}, $\Delta$ is an $A$-map, and by Lemma~\ref{16.2} it contains a cell $\Pi$ and a contiguity submap $\Gamma$ of $\Pi$ to $w$ such that $r(\Gamma)=0$ and $(\Pi,\Gamma,w)\ge\varepsilon$. Let $\mathbf{u}$ and $\mathbf{v}$ be the contiguity arcs of $\Gamma$ lying on $\partial\Pi$ and on $w$, respectively, and let $A^{n}$ or $A^{-n}$ be the label of $\partial\Pi$, where $A$ is a period of some rank $j$ and $n=n_A$; note that $A$ is cyclically reduced, being simple. Since $r(\Gamma)=0$, the labels of $\mathbf{u}$ and $\mathbf{v}$ are equal in the free group; as they are subwords of the cyclically reduced cyclic words $A^{\pm n}$ and $U$, they are graphically equal. Thus the cyclic word $U$ contains a subword $T=\varphi(\mathbf{v})$ that is a subword of the cyclic word $A^{n}$ or $A^{-n}$ and has length $|T|\ge\varepsilon n|A|$.
		
		We shall use the following observation. Let $\mathbf{v}_0$ be a subpath of $\mathbf{v}$, let $\mathbf{u}_0$ be the corresponding subpath of $\mathbf{u}$, so that $\varphi(\mathbf{u}_0)=\varphi(\mathbf{v}_0)$, and let $P$ be the label of $\partial\Pi$ read from the initial vertex of $\mathbf{u}_0$ in the direction of $\mathbf{u}_0$. Let $U_2$ be the cyclic shift of $U$ that begins with $\varphi(\mathbf{v}_0)$. Removing $\Pi$ and the $0$-cells of $\Gamma$ from $\Delta$, we obtain a diagram $\Delta_1$ whose contour label is freely equal to $P^{-1}U_2$. Since $\Delta_1$ has one cell of rank $j$ fewer than $\Delta$ and the same number of cells of every other positive rank, we have $\tau(\Delta_1)<\tau(\Delta)$.
		
		\textit{Case 1: $A=a$.} Then $n=mp$, and $T=a^{e}$ with $|e|\ge\varepsilon mp>2$ by \eqref{epsilon}. By Lemma~\ref{align}(2), $T$ lies in an $a$-syllable $a^{f}$ of the cyclic word $U$, where $f$ has the same sign as $e$ and $a^f$ is an $a$-block of $U$ or the union of the last and the first $a$-blocks. Let $U_1$ be the cyclic shift of $U$ that begins with this $a$-syllable. Then $U_1=Z^{-1}UZ$, where $Z$ is a product of complete blocks of $U$; hence $Z\in K$ and $U_1\in K_0$. Taking $\mathbf{v}_0=\mathbf{v}$, we have $P=a^{\operatorname{sgn}(e)mp}$. Since the letters of an $a$-syllable commute, it follows that the contour label of $\Delta_1$, read from a suitable vertex, is freely equal to
		$$U'=a^{-\operatorname{sgn}(e)mp}\,U_1\in K_0,$$
		and $\omega(U')=\omega(U_1)$ by Lemma~\ref{omega}(1); see Figure~\ref{fig:case1}.
		\begin{figure}[ht]
			\centering
			\begin{tikzpicture}[scale=0.9, every node/.style={font=\small}]
				\begin{scope}
					\coordinate (A) at (70:2);
					\coordinate (B) at (110:2);
					\fill[gray!20] (A) arc[start angle=70,end angle=110,radius=2] .. controls (-0.5,1.1) and (0.5,1.1) .. (A);
					\draw[thick] (0,0) circle[radius=2];
					\draw[thick] (B) .. controls (-0.5,1.1) and (0.5,1.1) .. (A);
					\draw[-{Stealth}] (-25:2.35) arc[start angle=-25,end angle=-65,radius=2.35];
					\node at (0,1.62) {$\Pi$};
					\node[above] at (90:2) {$a^{e}$};
					\node[above left] at (128:2) {$a^{h}$};
					\node[above right] at (52:2) {$a^{g}$};
					\node[below left] at (235:2) {$R$};
					\node at (0,0.9) {$C$};
					\node at (0,-0.6) {$\Delta$};
				\end{scope}
				\begin{scope}[xshift=6cm]
					\coordinate (A) at (70:2);
					\coordinate (B) at (110:2);
					\draw[thick] (A) arc[start angle=70,end angle=-250,radius=2];
					\draw[thick] (B) .. controls (-0.5,1.1) and (0.5,1.1) .. (A);
					\draw[-{Stealth}] (-25:2.35) arc[start angle=-25,end angle=-65,radius=2.35];
					\node[above left] at (128:2) {$a^{h}$};
					\node[above right] at (52:2) {$a^{g}$};
					\node[below left] at (235:2) {$R$};
					\node at (0,0.9) {$C^{-1}$};
					\node at (0,-0.6) {$\Delta_1$};
				\end{scope}
			\end{tikzpicture}
			\caption{Case~1 in the proof of Lemma~\ref{rw}. Left: the contour label of $\Delta$ is $U_1=a^{h}a^{e}a^{g}R$, where $a^{h}a^{e}a^{g}=a^{f}$ is the $a$-syllable containing $T=a^{e}$, and the label of $\partial\Pi$ is $TC=a^{\operatorname{sgn}(e)mp}$. Right: the contour label of $\Delta_1$ is $a^{h}C^{-1}a^{g}R=a^{h}a^{e-\operatorname{sgn}(e)mp}a^{g}R$, which is freely equal to $U'=a^{-\operatorname{sgn}(e)mp}U_1$.} 
			\label{fig:case1}
		\end{figure}
		
		\textit{Case 2: $A=b$.} Then $n=q$, and $T=b^{e}$ with $|e|\ge\varepsilon q>7$ by \eqref{epsilon}, which contradicts Lemma~\ref{align}(3). Thus, this case is impossible.
		
		\textit{Case 3: $j\ge2$.} Then $|A|=j\ge2$, and $A\notin\langle a\rangle$ (see Section~\ref{p2}), so $A$ contains a letter $b^{\pm1}$. Moreover, $n=\frac{mpq}{(\sigma_a(A),mp)}\ge q$, whence $|T|\ge\varepsilon q|A|>7|A|\ge2|A|+10$ by \eqref{epsilon}. We number the letters of $T$ by $1,2,\ldots,|T|$. Since $T$ is a subword of a periodic word with period $A^{\pm1}$, the $l$th and the $(l+|A|)$th letters of $T$ coincide for every $l$, and every subword of $T$ of length $|A|$ is a cyclic shift of $A^{\pm1}$; in particular, it contains a letter $b^{\pm1}$. Choose a letter $b^{\pm1}$ of $T$ in a position $j_0$ with $6\le j_0\le|A|+5$. By Lemma~\ref{align}(1), it belongs to a $y$-factor of $U$; since $j_0-5\ge1$ and $j_0+5\le|A|+10\le|T|$, this $y$-factor occupies the positions $i,\ldots,i+5$ of $T$ for some $i$ with $1\le i\le|A|+5$. As $i+|A|+5\le2|A|+10\le|T|$, the positions $i+|A|,\ldots,i+|A|+5$ of $T$ carry the same word $y^{\pm1}$, which is a $y$-factor of $U$ by Lemma~\ref{align}(4).
		
		Let $A_1$ be the subword of $T$ occupying the positions $i,\ldots,i+|A|-1$. Then $A_1$ is a cyclic shift of $A^{\pm1}$; it begins with a $y$-factor of $U$ and is followed in $U$ by a $y$-factor. Hence, every $a$-block of $U$ that meets $A_1$ lies in $A_1$, so $A_1$ is a product of $y$-factors and complete $a$-blocks, and $A_1\in K$. Write $\sigma_a(A_1)=pc$. Since $(\sigma_a(A),mp)=(pc,mp)=p\,(c,m)$, we have $n=dq$, where $d=m/(c,m)$. Set $B=A_1^{d}$; then $\sigma_a(B)=pcd\equiv0\pmod{mp}$, so $B\in K_0$.
		
		Let $\mathbf{v}_0$ be the subpath of $\mathbf{v}$ whose label is this occurrence of $A_1$. The label of $\partial\Pi$, read from the initial vertex of $\mathbf{u}_0$, is a periodic word of length $n|A|$ beginning with $A_1$; hence $P=A_1^{n}=B^{q}$. Let $U_2$ be the cyclic shift of $U$ that begins with this occurrence of $A_1$. Then $U_2=Z^{-1}UZ$, where $Z$ is a product of complete $a$-blocks and $y$-factors of $U$; hence $Z\in K$ and $U_2\in K_0$. The contour label of $\Delta_1$ is freely equal to
		$$U'=B^{-q}\,U_2\in K_0,$$
		and $\omega(U_2)=\omega(B)^{q}\,\omega(U')$; see Figure~\ref{fig:case3}.
		\begin{figure}[ht]
			\centering
			\begin{tikzpicture}[scale=0.9, every node/.style={font=\small}]
				\begin{scope}
					\coordinate (A) at (70:2);
					\coordinate (B) at (110:2);
					\fill[gray!20] (A) arc[start angle=70,end angle=110,radius=2] .. controls (-0.5,1.1) and (0.5,1.1) .. (A);
					\draw[thick] (0,0) circle[radius=2];
					\draw[thick] (B) .. controls (-0.5,1.1) and (0.5,1.1) .. (A);
					\draw[-{Stealth}] (-25:2.35) arc[start angle=-25,end angle=-65,radius=2.35];
					\node at (0,1.62) {$\Pi$};
					\node[above] at (90:2) {$A_1$};
					\node[below left] at (235:2) {$R$};
					\node at (0,0.9) {$C$};
					\node at (0,-0.6) {$\Delta$};
				\end{scope}
				\begin{scope}[xshift=6cm]
					\coordinate (A) at (70:2);
					\coordinate (B) at (110:2);
					\draw[thick] (A) arc[start angle=70,end angle=-250,radius=2];
					\draw[thick] (B) .. controls (-0.5,1.1) and (0.5,1.1) .. (A);
					\draw[-{Stealth}] (-25:2.35) arc[start angle=-25,end angle=-65,radius=2.35];
					\node[below left] at (235:2) {$R$};
					\node at (0,0.9) {$C^{-1}$};
					\node at (0,-0.6) {$\Delta_1$};
				\end{scope}
			\end{tikzpicture}
			\caption{Case~3 in the proof of Lemma~\ref{rw}. Left: the contour label of $\Delta$ is $U_2=A_1R$, and $\Pi$ is attached along this occurrence of $A_1$; the label of $\partial\Pi$ is $A_1C=A_1^{n}=B^{q}$, so $C=A_1^{n-1}$. Right: the contour label of $\Delta_1$ is $C^{-1}R=A_1^{1-n}R$, which is freely equal to $U'=B^{-q}U_2$.}
			\label{fig:case3}
		\end{figure}
		
		In Cases 1 and 3, let $U''$ be a cyclically reduced word in block form conjugate to $U'$ in $K$, say $U'=Z'U''Z'^{-1}$ with $Z'\in K$. Folding the contour of $\Delta_1$, we obtain a diagram with contour label $U''$ and the same cells. Hence $U''=1$ in $G$ and $\tau(U'')\le\tau(\Delta_1)<\tau(\Delta)=\tau(U)$. By the induction hypothesis, $\omega(U'')\in F_m^q$, and therefore $\omega(U')\in F_m^q$ by Lemma~\ref{omega}(2). Consequently, $\omega(U_1)=\omega(U')\in F_m^q$ in Case~1, and $\omega(U_2)=\omega(B)^q\omega(U')\in F_m^q$ in Case~3. Since $U_1$ and $U_2$ are conjugate to $U$ by elements of $K$, Lemma~\ref{omega}(2) yields $\omega(U)\in F_m^q$.
	\end{proof}
	
	\section{A free periodic subgroup of rank $m$ in $G$}\label{p4}
	
	It follows from \eqref{xyzy*} that the words $y_1, y_2, \ldots, y_m$ represent elements of the normal closure $\langle b\rangle^G$ of $b$ in $G$. Therefore, by Lemma~\ref{l3}, they generate a subgroup $H=\langle y_1,y_2,\ldots, y_m\rangle$ of $G$ of period $q$.
	
	\begin{lem}\label{HB}
		The elements $y_1,y_2,\ldots, y_m$ generate a subgroup $H$ of $G$ isomorphic to the free Burnside group of rank $m$ and period $q$: $H\cong B(m,q)$.
	\end{lem}
	\begin{proof}
		Consider the free Burnside group
		\begin{equation}\label{b3q}
			B(m,q)=\langle x_1, x_2, \ldots, x_m \mid X^q=1\rangle,
		\end{equation}
		where $X$ ranges over all words in the alphabet $\{x_1^{\pm1},x_2^{\pm1},\ldots,x_m^{\pm1}\}$. Since $H$ has period $q$, the map \eqref{perm} extends uniquely to a surjective homomorphism $\gamma\colon B(m,q)\to H$. We show that $\gamma$ is an isomorphism.
		
		Let $W\in F_m$ be such that $\gamma(W)=1$, that is, $\alpha(W)=1$ in $G$. The reduced form $U$ of $\alpha(W)$ is in block form (see Section~\ref{p3}) and represents the identity of $G$. By Lemma~\ref{rw}, $\omega(U)\in F_m^q$. Since $\omega(U)=\omega(\alpha(W))=W$ by Lemma~\ref{omega}(1), we obtain $W\in F_m^q$, that is, $W=1$ in $B(m,q)$.
	\end{proof}
	
	\begin{lem}\label{l6}
		\begin{enumerate}
			\item The inner automorphism $i_{a^p}$ of $G$ restricts to a regular automorphism of order $m$ of the subgroup $H=\langle y_1, y_2,\ldots, y_m\rangle$.
			\item The inner automorphism $i_{a}$ of $G$ restricts to a regular automorphism of order $pm$ of the subgroup $\langle b\rangle^G$.
		\end{enumerate}
	\end{lem}
	\begin{proof}
		(1) By Lemma~\ref{l1}, the element $a^p$ has order $m$ in $G$, and by Lemma~\ref{26.5} the centralizer of every nontrivial element of $G$ is cyclic. Hence the automorphism $i_{a^p}$ has order $m$. Since $i_{a^p}$ merely permutes cyclically the free generators of the free Burnside subgroup $H$, it restricts to an automorphism of $H$.
		
		It remains to show that $i_{a^p}$ has no nontrivial fixed points in $H$. Let $h\in H$ satisfy $i_{a^p}(h)=h$, that is, $h\in C_G(a^p)$. By Lemma~\ref{26.5}, the centralizer of every nontrivial element of $G$ is cyclic; hence $C_G(a^p)=\langle a\rangle$ by Lemma~\ref{26.4}(3). On the other hand, $H\subset\langle b\rangle^G$, and $\langle a\rangle\cap\langle b\rangle^G=\{1\}$ by Lemma~\ref{l2}. Therefore
		$$h\in C_G(a^p)\cap H\subset\langle a\rangle\cap\langle b\rangle^G=\{1\},$$
		that is, $h=1$. Thus $i_{a^p}$ has no nontrivial fixed points in $\langle b\rangle^G$, and in particular none in $H$.
		
		(2) The proof is the same as that of (1), with $a^p$ replaced by $a$.
	\end{proof}
	
	\section{Proof of Theorem~\ref{mpq}}\label{pmpq}
	
	\begin{proof}[Proof of Theorem~\ref{mpq}]
		We take as $G$ the group defined by \eqref{e1}. By Lemmas~\ref{l1} and~\ref{26.4}(2), $G$ has period $mpq$, and item~1 holds. Item~2 follows directly from Lemma~\ref{l3}, item~3 from Lemma~\ref{l6}(2), and item~4 is Lemma~\ref{HB}. Finally, item~5 follows from \eqref{xyzy} together with the fact that $a^p$ has order $m$.
	\end{proof}
	
	\section{Proof of Theorem~\ref{mq}}\label{pmq}
	
	\begin{proof}[Proof of Theorem~\ref{mq}]
		Take $H=\langle y_1, y_2,\ldots, y_m\rangle\le G$ as a model of $B(m,q)$: by Lemma~\ref{HB}, $H$ is isomorphic to the free Burnside group $B(m,q)$ of rank $m$ and period $q$. By Lemma~\ref{l6}(1), the inner automorphism $i_{a^p}$ of $G$ restricts to a regular automorphism of order $m$ of $H$; that is, it has no nontrivial fixed points in $H$. Moreover, by item~5 of Theorem~\ref{mpq}, conjugation by $a^p$ cyclically permutes the free generators of $H$ (see \eqref{xyz}). It remains to note that the prime $p>m$ and the number $q_m\ge pm$ must be chosen so large that the inequalities \eqref{epsilon} used in the proof of Lemma~\ref{rw} hold for all $q>q_m$; this is possible by the least parameter principle (see \cite[\S15.1]{O}).
	\end{proof}


\begin{thebibliography}{99}
		
		\bibitem{KN} V.~D.~Mazurov, E.~I.~Khukhro (eds.), \textit{The Kourovka
			Notebook: Unsolved Problems in Group Theory}, 17th ed., Institute of
		Mathematics, Novosibirsk, 2010.
		
		\bibitem{AtAs19} V.~S.~Atabekyan, H.~T.~Aslanyan, On fixed points of automorphisms, \textit{Proceedings of the YSU, Physical and Mathematical Sciences}, \textbf{53}:3 (2019), 147--149.
		
		\bibitem{A} S.~I.~Adian, \textit{The Burnside Problem and Identities in Groups}, Springer-Verlag, 1979.
		
		\bibitem{A82} S.~I.~Adian, Random walks on free periodic groups, \textit{Math. USSR-Izv.}, \textbf{21}:3 (1983), 425--434.
		
		\bibitem{Os} D.~V.~Osin, Uniform non-amenability of free Burnside groups, \textit{Arch. Math. (Basel)}, \textbf{88}:5 (2007), 403--412.
		
		\bibitem{At09u} V.~S.~Atabekyan, Uniform nonamenability of subgroups of free Burnside groups of odd period, \textit{Mat. Zametki}, \textbf{85}:4 (2009), 516--523; \textit{Math. Notes}, \textbf{85}:4 (2009), 496--502.
		
		\bibitem{At09m} V.~S.~Atabekyan, Monomorphisms of free Burnside groups, \textit{Math. Notes}, \textbf{86}:4 (2009), 457--462.
		
		\bibitem{O} A.~Yu.~Ol'shanskii, \textit{Geometry of Defining Relations in Groups}, Mathematics and its Applications (Soviet Series) 70, Kluwer Academic Publishers, Dordrecht, 1991.
		
		\bibitem{At87} V.~S.~Atabekyan, Simple and free periodic groups, \textit{Moscow Univ. Math. Bull.}, \textbf{42}:6 (1987), 80--82.
		
		\bibitem{At07} V.~S.~Atabekyan, On periodic groups of odd period $n\ge1003$, \textit{Math. Notes}, \textbf{82}:4 (2007), 443--447.
		
		\bibitem{O03} A.~Yu.~Olshanskii, Self-normalization of free subgroups in the free Burnside groups, in: \textit{Groups, Rings, Lie and Hopf Algebras}, Math. Appl., Vol. 555, Kluwer Academic, Dordrecht, 2003, pp.~179--187.
		
		\bibitem{At10} V.~S.~Atabekyan, The normalizers of free subgroups in free Burnside groups of odd period $n\ge1003$, \textit{Fundam. Prikl. Mat.}, \textbf{15}:1 (2009), 3--21; \textit{J. Math. Sci.}, \textbf{166}:6 (2010), 691--703.
		
		\bibitem{At09i} V.~S.~Atabekyan, On subgroups of free Burnside groups of odd exponent $n\ge1003$, \textit{Izv. Math.}, \textbf{73}:5 (2009), 861--892.
		
		\bibitem{I} S.~V.~Ivanov, On subgroups of free Burnside groups of large odd exponent, \textit{Illinois J. Math.}, \textbf{47}:1-2 (2003), 299--304.
		
		\bibitem{Sh} V. L. Shirvanyan, Embedding the group $B(\infty,n)$ in the group $B(2,n)$, \textit{Math. USSR-Izv.}, \textbf{10}:1 (1976), 181–-199.
		
		\bibitem{Ch} E.~A.~Cherepanov, Normal automorphisms of free Burnside groups of large odd exponents, \textit{Int. J. Algebra Comput.}, \textbf{16}:5 (2006), 839--847.
		
		\bibitem{At11} V.~S.~Atabekyan, Normal automorphisms of free Burnside groups, \textit{Izv. Math.}, \textbf{75}:2 (2011), 223--237.
		
		\bibitem{At13} V.~S.~Atabekyan, Splitting automorphisms of free Burnside groups, \textit{Sb. Math.}, \textbf{204}:2 (2013), 182--189.
		
		\bibitem{At14} V.~S.~Atabekyan, Splitting automorphisms of order $p^k$ of free Burnside groups are inner, \textit{Math. Notes}, \textbf{95}:5 (2014), 586--589.
		
		\bibitem{At13p} V.~S.~Atabekyan, The automorphism tower problem for free periodic groups, \textit{Proceedings of the YSU, Physical and Mathematical Sciences}, 2013, no.~2, 3--7.
		
		\bibitem{At13ij} V.~S.~Atabekyan, The groups of automorphisms are complete for free Burnside groups of odd exponents $n\ge1003$, \textit{International Journal of Algebra and Computation}, \textbf{23}:6 (2013), 1485--1496.
		
		\bibitem{At15} V.~S.~Atabekyan, The automorphisms of endomorphism semigroups of free burnside groups, \textit{International Journal of Algebra and Computation}. \textbf{25}:4 (2015) 669--674.
		
		\bibitem{F75} J.~Dyer and E.~Formanek, The automorphism group of a free group is complete,
		\textit{J. London Math. Soc.}, \textbf{11}:2 (1975) 181--190.
		
		\bibitem{F02} E.~Formanek, A question of B. Plotkin about the semigroup of endomorphisms of a free group, \textit{Proc. Amer. Math. Soc.}, 130 (2002) 935--937.
		
		\bibitem{Cou} R.~Coulon, Outer automorphisms of free Burnside groups, \textit{Comment. Math. Helv.}, \textbf{88}:4 (2013), 789--811.
		
		\bibitem{Is} S.~V.~Ivanov, Strictly verbal products of groups and A.~I.~Mal'tsev's problem on operations over groups, \textit{Trans. Moscow Math. Soc.}, \textbf{54}, (1993), 217--249.
		
	\end{thebibliography}
\end{document}